\documentclass[11pt,oneside]{amsart}
\usepackage{amsfonts}
\usepackage{amsmath}
\usepackage{amssymb}
\usepackage{amsthm}
\usepackage{mathrsfs}
\usepackage[bookmarksnumbered]{hyperref}
\usepackage[all]{xy}
\SelectTips{cm}{}

\title{Homological Localisation of Model Categories}

\author{David Barnes}
\author{Constanze Roitzheim}
\address{D. Barnes\\ Department of Pure Mathematics\\
University of Sheffield \\ Hicks Building\\ Hounsfield Road\\
Sheffield S3 7RH, UK}
\address{C. Roitzheim \\ University of Kent \\ School of Mathematics, Statistics and Actuarial Science\\ Cornwallis Building\\ Canterbury, Kent, CT2 7NF,
UK}
\thanks{The first author was supported by EPSRC grant EP/H026681/1, the second author by EPSRC grant EP/G051348/1. }
\date{$8^\text{th}$ December 2012}

\DeclareMathOperator{\Ho}{Ho}

\DeclareMathOperator{\Hom}{Hom}
\DeclareMathOperator{\sset}{sSet}

\DeclareMathOperator{\Ch}{Ch}

\DeclareMathOperator{\map}{map}
\DeclareMathOperator{\Map}{Map}
\DeclareMathOperator{\id}{Id}
\DeclareMathOperator{\leftmod}{--mod}
\DeclareMathOperator{\rightmod}{mod--\!}

\newcommand{\lradjunction}{\,\,\raisebox{-0.1\height}{$\overrightarrow{\longleftarrow}$}\,\,}
\newcommand{\rladjunction}{\,\,\raisebox{-0.1\height}{$\overleftarrow{\longrightarrow}$}\,\,}

\newcommand{\ncal}{\mathcal{N}}
\newcommand{\mcal}{\mathcal{M}}
\newcommand{\ccal}{\mathcal{C}}
\newcommand{\dcal}{\mathcal{D}}
\newcommand{\ecal}{\mathcal{E}}
\newcommand{\gcal}{\mathcal{G}}
\newcommand{\Sp}{\mathcal{S}}
\newcommand{\dsmash}{\wedge^L}
\newcommand{\smashprod}{\wedge}
\newcommand{\co}{\colon \!}
\newcommand{\Pu}{\mathcal{P}}
\newcommand{\fibrep}{\hat{f}}

\newtheorem{theorem}{Theorem}[section]
\newtheorem{proposition}[theorem]{Proposition}
\newtheorem{corollary}[theorem]{Corollary}
\newtheorem{lemma}[theorem]{Lemma}

\newtheorem{definition}[theorem]{Definition}

\newtheorem{rmk}[theorem]{Remark}

\newtheorem*{thm}{Theorem}

\begin{document}
%\linenumbers

\begin{abstract}
\noindent One of the most useful methods for studying the 
stable homotopy category is localising at some spectrum $E$. 
For an arbitrary stable model category we introduce a candidate 
for the $E$-localisation of this model category. 
We study the properties of this new construction 
and relate it to some well-known categories. 
\end{abstract}
\maketitle

\section*{Introduction}

The stable homotopy category is spectacularly complicated
yet of fundamental importance to homotopy theorists. 
A standard and highly successful method of dealing with this complexity is to 
``filter out'' some of this information via a Bousfield localisation. 
In return we obtain a more structured category with useful and interesting patterns. 

More precisely, we choose some homology theory $E_*$ and replace the stable homotopy category $\Ho( \Sp)$
with $\Ho(L_E \Sp)$, the full subcategory of 
$\Ho( \Sp)$ with objects the $E$-local spectra. 
This means that in the passage from $\Ho(\Sp)$ to $\Ho(L_E\Sp)$, the $E_*$-isomorphisms are formally inverted.
Bousfield's paper \cite{Bou79} is the original source of this idea.

There are a number of other categories that share of the properties
$\Ho( \Sp)$. It would be advantageous if we could generalise
the specialised form of ``homological'' Bousfield localisation to these other contexts, namely arbitrary stable model categories. A comprehensive general study of Bousfield localisation of model categories
is given in \cite{hir03}, but we are specifically interested in the construction of a \textbf{homological localisation} of a stable model category. 

The main motivation comes again from the study of the stable homotopy category. In order to understand spectra, $\Ho(\Sp)$ and its various $E$-localisations it is necessary to relate $\Sp$ and $L_E\Sp$ to other stable model categories $\ccal$. For example, one can study to what extent there is a stable model category $\ccal$ whose homotopy category ``models'' $\Ho(L_E\Sp)$ and how similar $\ccal$ is to $L_E\Sp$ in terms of higher homotopy behaviour. To make those links it would be a desirable tool to have the corresponding $E$-localisations of $\ccal$ in order to compare $E$-local spectra to other counterparts related to $\ccal$.

A stable model category $\ccal$ is a model category 
whose associated homotopy category $\Ho(\ccal)$ 
is triangulated via the construction of
\cite[Section 7]{Hov99}. Lenhardt proved in \cite{Len12}
that $\Ho(\ccal)$ is a module over $\Ho(\Sp)$ whenever
$\ccal$ is a stable model category. Hence we have a tensor product
\[
- \smashprod^L - \co \Ho(\ccal) \times \Ho(\Sp)
\longrightarrow 
\Ho(\ccal)
\]
and an enrichment of $\Ho(\ccal)$ in $\Ho(\Sp)$. This technique is called \textbf{stable frames}.

Using this action on the homotopy category of a stable model category 
one could try to make a new model structure on $\ccal$ such that the 
weak equivalences are the ``$E_*$-isomorphisms'': those maps
$f \co X \to Y$ in $\ccal$ such that 
\[
f \smashprod ^L \id \co X \smashprod^L E \to Y \smashprod^L E 
\]
is an isomorphism in $\Ho(\ccal)$. Such a model structure would 
deserve the name $L_E \ccal$. 
But it seems particularly difficult 
to check that this construction exists for general $\ccal$.
For spectra, the argument appears in \cite[Section VIII.1]{EKMM}
and requires numerous unpleasant cardinality arguments. 

For well-behaved stable model categories $\ccal$
we are going to produce a new model structure $\ccal_E$ that avoids such set-theoretic awkwardness. This $\ccal_E$ is a good replacement for the $E$-localisation of $\ccal$ because of the following universal property: $\ccal_E$ is the ``closest'' model category to $\ccal$ such that any Quillen adjunction from spectra to $\ccal$ 
\[
\Sp \lradjunction \ccal
\]
gives rise to a Quillen adjunction $$L_E\Sp \lradjunction \ccal_E$$ from $E$-local spectra to $\ccal_E$ We are also able to give another description of $\ccal_E$
in terms of pushouts of model categories, which shows 
how strong the universal property of this new model structure is. 

We are also able to give an improvement
of \cite[Theorem 9.5]{BarRoi11b}: we can show that for all $E$, the homotopy information of $E$-local spectra is entirely encoded in the $\Ho(\Sp)$-module structure on the $E$-local stable homotopy category. This was previously only possible with the strong restriction that $E$ is smashing. 
Hence we have the following, which appears as Theorem \ref{thm:elocalrigid}.

\begin{thm}
Let $\ccal$ be a stable model category.
Assume we have an equivalence of triangulated categories
\[
\Phi: \Ho(L_E\Sp) \longrightarrow \Ho(\ccal)
\]
then $L_E\Sp$ and $\ccal$ are Quillen equivalent if and only if $\Phi$ is an equivalence of $\Ho(\Sp)$-model categories. 
\end{thm}

\bigskip
\subsection*{Organisation}

Firstly, we recall some definitions and conventions regarding Bousfield localisation and stable frames. We also re-introduce the concept of \textbf{stably $E$-familiar model categories}: in \cite{BarRoi11b} we studied those $\ccal$ such that the action of $\Ho(\Sp)$ factors over the functor
$\Ho(\Sp) \to \Ho(L_E \Sp)$. In particular the homotopy category of such a model category has an enrichment in the more structured category $\Ho(L_E \Sp)$. We called such categories stably $E$-familiar.

We then turn to the question of altering a model structure on a given category
so as to obtain a stably $E$-familiar model category.  
In Section \ref{sec:spectralfamiliarisation} 
we consider the simpler case of spectral model categories: 
such a model category is defined in a similar way to a simplicial model category,
but with simplicial sets replaced by the model category of symmetric spectra.
We construct the \textbf{stable $E$-familiarisation} of a spectral model category 
in this section. 

In Section \ref{sec:stablefamiliarisation} we extend our results to more general stable model categories. We prove that the \textbf{stable $E$-familiarisation} 
of a model category $\ccal$ is the closest stably $E$-familiar
model category to $\ccal$ in the following sense. 
The result below also implies that our construction has the universal 
property we described earlier.

\begin{thm}
Let $\ccal$ be a stable, proper and cellular model category such that the domains of the generating cofibrations of $\ccal$ are cofibrant. 
Then there is a model structure $\ccal_E$ on $\ccal$ such that 
\begin{enumerate}
\item $\ccal_E$ is stably $E$-familiar, 
%\item the weak equivalences of $\ccal_E$ are the 
%$T'$-equivalences, for $T'$ the class below  
%\[
%T'=\{ X \smashprod^L f \,\,|\,\, X \in \ccal, \, \, f \mbox{ is an $E$-equivalence of spectra}\}
%\]
\item if $F \co \ccal \to \ecal$ is a left Quillen functor and $\ecal$ is stably $E$-familiar, then $F$ factors over $\ccal \to \ccal_E$.
\end{enumerate}
\end{thm}

Section \ref{sec:examples} consists of several examples of $\ccal_E$ for some $E$ and $\ccal$ involving algebraic model categories, chromatic localisations and module categories over a ringoid spectrum. 

In Section \ref{sec:pushouts} we rephrase the universal property of $\ccal_E$ in terms of homotopy pushouts of model categories. 

Finally, we prove a full version of the \textbf{modular rigidity} theorem that all homotopy information of $E$-local spectra is governed by the $\Ho(\Sp)$-action on $\Ho(L_E\Sp)$ given by framings. 

%and 
%is `closer' to $\ccal$ in the sense that if 
%this $L_E \ccal$ exists, 
%the functor $\Ho(\ccal) \to \Ho(L_E \ccal)$
%would factor over $\Ho(\ccal) \to \Ho(\ccal_E)$. 
%Indeed, when localisation at $E$ is smashing, 
%we can show that $\ccal_{L_E S^0}$ would equal 
%$L_{L_E S^0} \ccal$ and hence the model structure
%$L_{L_E S^0} \ccal$ exists.

\section{Bousfield localisation}\label{sec:localisation}

We begin with an introduction to Bousfield localisation
at a homology theory $E$. Throughout the paper
when we refer to spectra, we mean symmetric spectra \cite{hss00} unless stated otherwise.

Let $E$ be a spectrum. Then $E$ corepresents a homology functor $E_*$ on the category of spectra via $$E_*(X) = \pi_*(E \smashprod X).$$ Bousfield used this to construct a homotopy category of spectra where maps which induce isomorphisms on $E_*$--homology are isomorphisms \cite{Bou79}. We recap some of the definitions from this work.  We denote homotopy classes of maps of spectra by $[-,-]$.

\begin{definition}
A map $f \co X \to Y$ of spectra is an \textbf{$E$--equivalence} if $E_*(f)$ is an isomorphism.
A spectrum $Z$ is \textbf{$E$--local} if $f^* \co [Y,Z] \to [X,Z]$ 
is an isomorphism for
all $E$--equivalences $f \co X \to Y$. A spectrum $A$ is \textbf{$E$--acyclic} if
$[A,Z]=0$ for all $E$--acyclic $Z$.
An $E$--equivalence from $X$ to an $E$--local object $Z$ is called an \textbf{$E$--localisation}.
\end{definition}

Bousfield localisation of spectra gives a homotopy theory that is particularly sensitive towards $E_*$ and $E$--local phenomena. The $E$--local homotopy theory is obtained from the category of spectra by formally inverting the $E$--equivalences.

This can be seen as a special case of a more general result by Hirschhorn.
For $X, Y \in \ccal$, we let $\map_\ccal(X,Y) \in \sset$ denote the homotopy function object, see \cite[Chapter 17]{hir03} and Section 
\ref{sec:framings}.

\begin{definition}\label{def:local}
Let $S$ be a set of maps in $\ccal$. Then an object $Z \in \ccal$ is \textbf{$S$-local} if
\[
\map_\ccal(s,Z): \map_\ccal(B,Z) \longrightarrow \map_\ccal(A,Z)
\]
is a weak equivalence in simplicial sets for any
$s: A \longrightarrow B$ in $S$.
A map $f:X \longrightarrow Y \in \ccal$ is an \textbf{$S$-equivalence} if
\[
\map_\ccal(f,Z): \map_\ccal(Y,Z) \longrightarrow \map_\ccal(X,Z)
\]
is a weak equivalence for any $S$-local $Z \in \ccal$.
An object $W \in \ccal$ is \textbf{$S$-acyclic} if
$$
\map_\ccal(W,Z) \simeq *
$$
for all $S$-local $Z \in \ccal$. 
\end{definition}

A \textbf{left Bousfield localisation} of a model category $\ccal$ with respect to a class of maps $S$ is a new model structure $L_S\ccal$ on $\ccal$ such that
\begin{itemize}
\item the weak equivalences of $L_S\ccal$ are the $S$-equivalences, 
\item the cofibrations of $L_S \ccal$ are the cofibrations of $\ccal$,
\item the fibrations of $L_S\ccal$ are those maps that have the right lifting property with respect to cofibrations that are also $S$-equivalences. 
\end{itemize}

Hirschhorn proves that if $S$ is a set and $\ccal$ is left proper and cellular
then $L_S\ccal$ exists. (We will give rough definitions of these two terms below.) 
In the case of localising spectra at a homology theory one wants to invert 
the class of $E_*$-isomorphisms. Since this is not a set, 
one cannot use Hirschhorn's result directly. 
In \cite[Section VIII.1]{EKMM} it is shown that there is a set $S$ whose 
$S$-equivalences are exactly the $E_*$-isomorphisms. 
Hence, the key to proving the existence of homological localisations is to 
find a set giving the correct notion of equivalence. We shall encounter this
idea again when constructing $\ccal_E$.

A model category is \textbf{left proper} if
the pushout of a weak equivalence along a cofibration is 
a weak equivalence. A model category is \textbf{right proper} if 
the pullback of a weak equivalence along a fibration is a 
weak equivalence. If a model category is both left and right proper, 
we say that it is \textbf{proper}. 

We also a need a stronger version of cofibrantly generated, 
which forces cell complexes to be better behaved. 
The actual definition is technical and not particularly illuminating, 
so we shall simply say that a model category is \textbf{cellular} 
if it is cofibrantly generated by sets $I$ and $J$, 
and the domains and codomains of $I$ and $J$ satisfy some nice
cardinality conditions. We leave the details to 
\cite[Definition 12.1.1]{hir03}.

\section{Stable framings}\label{sec:framings}

Framings are a powerful tool that describe and classify Quillen functors from simplicial sets or spectra to arbitrary model categories. They were first developed by Hovey in \cite[Section 5.2]{Hov99}. For a model category $\ccal$, he investiagtes cosimplicial and simplicial resolutions of objects in $\ccal$. These are called ``frames''. In more detail, a frame of an object $A \in \ccal$ is a cofibrant replacement of the constant cosimplicial object $A \in \ccal^\Delta$ in the Reedy model category of cosimplicial objects in $\ccal$. From these notions one obtains bifunctors
\[
\begin{array}{r@{\  : \ }ll}
- \otimes - & \ccal \times \sset \longrightarrow \ccal,  \\
\map_l(-,-) & \ccal^{op} \times \ccal \longrightarrow \sset,   \\
(-)^{(-)} &  \ccal \times \sset^{op} \longrightarrow \ccal,  \\
\map_r(-,-) & \ccal^{op} \times \ccal \longrightarrow \sset
\end{array}
\]
satisfying certain adjunction properties. The notation $\otimes$ stems from the fact that $A \otimes \mathbb{S}^0 \simeq A.$ However, this set-up does not equip $\ccal$ with the structure of a simplicial model category because the ``mapping spaces'' $\map_l(X,Y)$ and $\map_r(X,Y)$ only agree up to a zig-zag of weak equivalences for $X, Y \in \ccal$ \cite[Proposition 5.4.7]{Hov99}. But their derived functors agree, leaving us with the following \cite[Theorem 5.5.3]{Hov99}.

\begin{theorem}[Hovey]
Let $\ccal$ be any model category. Then its homotopy category $\Ho(\ccal)$ is a closed $\Ho(\sset)$-module category.
\end{theorem}

In particular, this equips any model category with the notion of a homotopy mapping space.
Moreover, framings satisfy the following important properties.
\begin{itemize}
\item If $\ccal$ carries the structure of a simplicial model category \cite[Definition 4.2.18]{Hov99}, then the two $\Ho(\sset)$-module structures coming from either framings or the simplicial structure agree \cite[Theorem 5.6.2]{Hov99}.
\item If $F: \sset \longrightarrow \ccal$ is a left Quillen functor with $F(\mathbb{S}^0)=A$, then the left derived functors of $F$ and of  the framing functor $A \otimes -: \sset \longrightarrow \ccal$ agree. Thus, every left Quillen functor from simplicial sets to any model category can be described, up to homotopy, by a frame. 
\end{itemize}

The second property follows from the fact that the category of cosimplicial objects $\ccal^\Delta$ is equivalent to the category of adjunctions $\sset \lradjunction \ccal$. A cosimplicial object $A^\bullet$ corresponds to a Quillen adjunction under this equivalence if and only if it is a frame, that is  $A^\bullet$ is cofibrant and homotopically constant, \cite[Proposition 3.2]{BarRoi11b}.

\bigskip
In \cite{Len12} Fabian Lenhardt described an analogous set-up for spectra and stable model categories. Now let $\ccal$ be a stable model category. First, Lenhardt shows that the category of adjunctions between spectra and a stable model category $\ccal$ is equivalent to the category of ``$\Sigma$-cospectra'' $\ccal^\Delta(\Sigma)$. An object in $\ccal^\Delta(\Sigma)$ consists of a sequence of cosimplicial objects $X_n \in \ccal^\Delta$ together with structure maps $$\Sigma X_{n+1} \longrightarrow X_n.$$ He then characterises those $\Sigma$-cospectra that give rise to Quillen adjunctions under this equivalence, calling them stable frames. These give rise to bifunctors $ - \wedge -$ and $\Map(-,-)$ satisfying the expected adjunction properties.

As in the unstable case, this is not rigid enough to equip any stable model category $\ccal$ with the structure of a spectral model category. However, the above bifunctors give rise to the following \cite[Theorem 6.3]{Len12}.

\begin{theorem}[Lenhardt]
Let $\ccal$ be a stable model category. Then $\Ho(\ccal)$ is a closed $\Ho(\Sp)$-module category.
\end{theorem}

Note that unlike in the stable case, certain technical subtleties only allow the construction of a $\Ho(\Sp)$-module structure on $\Ho(\ccal)$ rather than a closed module structure.

As expected, this satisfies the following key properties.
\begin{itemize}
\item If $\ccal$ is already a spectral model category, then the $\Ho(\Sp)$-module structure derived from the spectral structure agrees with the $\Ho(\Sp)$-module structure coming from stable frames \cite[Example 6.7]{BarRoi11b}.
\item By construction, every left Quillen functor $F: \Sp \longrightarrow \ccal$ is, up to homotopy, of the form $X \smashprod -: \Sp \longrightarrow \ccal$ for some fibrant-cofibrant $X \in \ccal$.
\item In particular, for any fibrant-cofibrant $X \in \ccal$ there is a left Quillen functor $\Sp \longrightarrow \ccal$ that sends the sphere spectrum to $X$. 
\item Any stable frame and thus any Quillen functor $\Sp \longrightarrow \ccal$ is, up to homotopy, entirely determined by its image on the sphere.
\end{itemize}

As we have already mentioned, the homotopy theory of $L_E \Sp$ is often much better understood than $\Sp$. So it is worth asking if some stable model categories 
have more in common with $L_E \Sp$ than $\Sp$. 
We answer this question and obtain several useful results using this idea in 
\cite{BarRoi11b}. We give the fundamental definitions below.

\begin{definition}
We say that a stable frame $X \in \ccal^\Delta(\Sigma)$ is an \textbf{$E$-local frame} 
if it gives rise to a Quillen functor pair 
\[
X \wedge - : L_E\Sp \lradjunction \ccal: \Map(X,-)
\]
A stable model category $\ccal$ is \textbf{stably $E$-familiar} if every stable frame 
is an $E$-local frame.
\end{definition}

This is \cite[Definition 7.1]{BarRoi11b}. This generalises the notion of an $L_E\Sp$-model category in the following sense: if $\ccal$ is already an $L_E\Sp$-model category, then the $\Ho(L_E\Sp)$-module structure on $\Ho(\ccal)$ agrees with the $\Ho(L_E\Sp)$-module structure given by $E$-local frames \cite[Proposition 7.6]{BarRoi11b}. We can further characterise stably $E$-familiar model categories as follows \cite[Theorem 7.8]{BarRoi11b}.

\begin{theorem}
Let $\ccal$ be a stable model category. Then $\ccal$ is stably $E$-familiar if and only if the homotopy mapping spectrum $\mathbb{R} \Map_\ccal(X,Y)$ is an $E$-local spectrum for all $X, Y \in \ccal$.
\end{theorem}

\bigskip

We can use the theory of $E$-local framings to study algebraic model categories. An algebraic model category is a $\Ch(\mathbb{Z})$-model category in the sense of \cite[Definition 4.2.18]{Hov99}. Thus a $\Ch(\mathbb{Z})$-model category is enriched, tensored and cotensored over chain complexes and satisfies the $\Ch(\mathbb{Z})$-analogue of the compatibility axiom (SM7). This implies that the homomorphism spectra obtained via framings are products of Eilenberg-MacLane spectra \cite[Proposition II.2.20]{GJ99}, \cite[Section 2.6]{DugShi07}. Using the computations of Guti{\'e}rrez in \cite{Gut10} one can draw the following conclusions \cite[Section 9]{BarRoi11b}.
\begin{itemize}
\item For $n \ge 1$ and there are no algebraic stably $K(n)$-familiar model categories, where $n$ denotes the $n^{th}$ Morava-$K$-theory.
\item Let $E(n)$ denote the $n^{th}$ chromatic Johnson-Wilson spectrum. An algebraic model category is stably $E(n)$-familiar if and only if it is rational.
\item The category $L_E\Sp$ of $E$-local spectra is algebraic if and only if $E=H\mathbb{Q}$.
\end{itemize}

Now we turn to the question of whether any model category can be 
made stably $E$-familiar in some natural way.

\section{\texorpdfstring{$E$--}{E-}Familiarisation of spectral model categories}\label{sec:spectralfamiliarisation}

For any homology theory $E$ we can consider the category of $E$-local spectra, $L_E\Sp$. 
Hence we would like to know if a reasonable notion of $E$-localisation exists for an arbitrary stable model category $\ccal$.

Intuitively, a promising definition would be a Bousfield localisation $L_E\ccal$ of $\ccal$ where one localises at the class of ``$E$-equivalences'' given by 
\[
\{f: X \longrightarrow Y \in \ccal \,\,|\,\, f \wedge^L E: X \wedge^L E \longrightarrow Y \wedge^L E \,\,\,\mbox{is a weak equivalence in $\ccal$}\}
\]
where the action $\wedge$ of a spectrum on an element of $\ccal$ is defined via stable frames. However, showing the existence of Bousfield localisations at a class of maps is set-theoretically awkward. The standard method to circumvent this difficulty is to find a set of maps $S$ such that the $S$-equivalences are precisely the $E$-equivalences. This is an extremely difficult task, see 
\cite[Section VIII.1]{EKMM}, so it is not clear if a good notion of $E$-localisation exists for general model categories. 

Instead, we will construct the \emph{stable $E$-familiarisation} of $\ccal$ which is the ``closest'' $E$-familiar model category to $\ccal$. 
We will then draw some conclusions about its properties which will show that this construction is the right choice for an analogue of $E$-localisation for general stable $\ccal$. For example, the first theorem will show that every Quillen adjunction
\[
\Sp \lradjunction \ccal
\]
will give rise to a Quillen adjunction
\[
L_E\Sp \lradjunction \ccal_E.
\]
The first question to answer is: 
what kind of maps to we want to invert in order to construct $\ccal_E$?
In a stably $E$-familiar model category $\dcal$
any map of the form 
\[
\id \smashprod^L j \co X \smashprod^L A \to X \smashprod^L B 
\]
for $j \co A \to B$ an $E$-equivalence of spectra
and $X \in \dcal$ is a weak equivalence. 
Hence we could try to localise $\ccal$ at this
class of maps. So we must find some set of maps 
$S$ such that the $S$-equivalences equals this class. 

We need a couple of technical results first. For this section 
we shall work with $\Sp$-model 
categories in the 
sense of \cite[Definition 4.2.18]{Hov99}, where $\Sp$ again denotes the model category of symmetric spectra.
Such a model category $\dcal$ is enriched, tensored and cotensored over 
symmetric spectra in simplicial sets and satisfies the 
appropriate analogue of Quillen's SM7 axiom for simplicial
model categories. We shall refer to $\dcal$ as being a 
\textbf{spectral model category}. We may also talk about 
$L_E \Sp$-model categories, where we use the 
$E$-local model structure on $\Sp$. 
A spectral model category is in particular stable and simplicial, 
see \cite[Lemma 3.5.2]{ss03stabmodcat}. We will see later that the restriction to spectral model categories is not as big a restriction as it might seem at first.

We denote the pushout-product of two maps by $\square$, so for $f:X \longrightarrow Y$ and $g: A \longrightarrow B$ the pushout-product of $f$ and $g$ is
\[
f \square g: X \smashprod B \coprod_{X \smashprod A} Y \smashprod A \longrightarrow Y \smashprod B.
\]

Recall that a set of maps $S$ in a stable model category $\dcal$
is said to be \textbf{stable} if the class of $S$-local
objects is closed under suspension. 
By \cite[Proposition 3.6]{BarRoi12} if $\dcal$ and $S$ are stable then
so is $L_S \dcal$. 

\begin{proposition}\label{prop:spectral}
Let $\dcal$ be a left proper, cellular and spectral model category 
Let $S$ be a stable set of maps in $\dcal$. 
Then $L_S \dcal$ is also a spectral model category.  
\end{proposition}

\begin{proof}
Since $\dcal$ is left proper and cellular, $L_S \dcal$ exists. 
We must prove that if $i$ is a cofibration of $L_S \dcal$ and 
$j$ is a cofibration of $\Sp$ then 
$i \square j$ is a cofibration of $L_S \dcal$
that is a weak equivalence (in $L_S \dcal$) if either of 
$i$ or $j$ is. 
Since $\dcal$ is spectral and the cofibrations are 
unchanged by left Bousfield localisation, we know that 
$i \square j$ is a cofibration whenever $i$ and $j$ are. 
Furthermore if $j$ is an acyclic cofibration of symmetric spectra, 
then $i \square j$ is a weak equivalence in $\dcal$ and hence it is also
an $S$-equivalence. 

The third case is where $i$ is an acyclic cofibration of 
$L_S \dcal$ and $j$ is a cofibration of symmetric spectra. 
We must show that $i \square j$ is an $S$-equivalence. 
By \cite[Lemma 4.2.4]{Hov99} it suffices to prove this for
$j$ a generating cofibration of symmetric spectra and $i$ a generating
acyclic cofibration of $L_S \dcal$. 
Hence we may assume that 
$j$ has form $$F_n K \to F_n L$$ where 
$F_n$ is the left adjoint to evaluation at level $n$, and 
$K$ and $L$ are simplicial sets. 
By \cite[Proposition 4.5.1]{hir03} the domain of $j$ is cofibrant,
so it follows that both the domain and codomain of 
$i \square j$ are cofibrant. 
The set $S$ is stable, so the class of $S$-equivalences in $\Ho(\dcal)$ 
is closed under suspension and desuspension.  
Thus $i \square j$ is an $S$-equivalence if and only if 
$$\Sigma^n (i \square j) \cong i \square \Sigma^n j$$
is an $S$-equivalence for all $n$. Note that we do not need to use the 
derived functor of $\Sigma$ in that statement since
all the terms are cofibrant. 

We know that $\Sigma^n F_n K$ is weakly equivalent to 
$F_0 K$ in $\Sp$. Hence for any cofibrant 
$X \in \dcal$, $$X \smashprod \Sigma^n F_n K \to 
X \smashprod F_0 K$$ is a weak equivalence of $\dcal$. 
We also know that the domains of the maps 
$i \square \Sigma^n j$ and $i \square (F_0 K \to F_0 L)$ are
pushouts of cofibrations between cofibrant objects. 
It follows that $i \square \Sigma^n j$ is weakly equivalent 
to the map $i \square (F_0 K \to F_0 L)$. 
The bifunctor $$- \smashprod F_0 - \co \dcal \times \sset \to \dcal$$
gives $\dcal$ the structure of a simplicial model 
category. We may now use \cite[Theorem 4.1.1]{hir03}, 
which states that since $\dcal$ is simplicial, so is 
$L_S \dcal$. Consequently we see see that $i \square (F_0 K \to F_0 L)$
is an $S$-equivalence. Hence $i \square j$ is also an $S$-equivalence
and $L_S \dcal$ is a spectral model category. 
\end{proof}

\begin{proposition}\label{prop:espectral}
Let $\dcal$ be a left proper, cellular and spectral model category 
with generating cofibrations $I_\dcal$ and 
generating acyclic cofibrations $J_\dcal$.
Let $J_E$ be the set of generating cofibrations for 
$L_E \Sp$.
Define 
\[
S= I_{\dcal} \square J_E = 
\{
i \square j \ | \ i \in I_{\dcal}, \ \ j \in J_E
\}.
\] 
Then $L_S \dcal$ is an $L_E \Sp$-model category
and hence is stably $E$-familiar.  
\end{proposition}

\begin{proof}
The set $J_E$ is closed under desuspension in the sense that for 
any element $j \in J_E$ there is an element $j'$ with 
$\Sigma j' \simeq j$. It follows that the same holds for 
$S$, so it is stable in the sense of \cite[Definition 3.2]{BarRoi12}. 
Thus $L_S \dcal$ is also a stable model category. 
By Lemma \ref{prop:spectral} it is also a $\Sp$-model category. 

To see that it is an $L_E \Sp$-model category 
we only need check that if $i$ is a cofibration of $L_S \dcal$ and 
$j$ is an acyclic cofibration of $L_E \Sp$ then 
$i \square j$ is an $S$-equivalence. 
By \cite[Lemma 4.2.4]{Hov99} it suffices to prove this for
$i \in I_\dcal$ and $j \in J_E$. But then $i \square j$
is an element of $S$ and hence is an $S$-equivalence. 
\end{proof}

We now show that this set $S$ has the correct homotopical behaviour in terms of $E$-familiarity by 
giving another description of the weak equivalences
of $L_S \dcal$. 

\begin{proposition}\label{prop:classofequivalences}
Let $\dcal$ be a left proper, cellular, spectral model category, 
such that the domains of the generating cofibrations 
of $\dcal$ are cofibrant. 
Let $T$ be the class of maps
\[
T=\{ X \smashprod^L f \,\,|\,\, X \in \dcal, \, \, f \mbox{ is an $E$-equivalence of spectra}\}.
\]
Then the class of $T$-equivalences is equal to the class of 
$S$-equivalences.
\end{proposition}

\begin{proof}
Take some cofibrant $X \in \dcal$. Then the functor
$$X \smashprod - \co L_E \Sp \to L_S \dcal$$ 
is a left Quillen functor by 
Proposition \ref{prop:espectral}. 
Hence $X \smashprod -$ takes $E$-equivalences
between cofibrant spectra to $S$-equivalences. 
Thus every element of $T$ is a weak equivalence in $L_S \dcal$. 

Now we will show that every element of $S$ is also a $T$-equivalence.
Consider $i \square j \in S$ for 
$i \co X \to Y$ a generating cofibration of $\dcal$ and 
$j \co A \to B$ a generating acyclic cofibration for $L_E \Sp$. 
Since $X$, $Y$, $A$ and $B$ are all cofibrant, the maps 
$X \smashprod j$ and $Y \smashprod j$ are in the class $T$. 
Let $P$ be the domain of $i \square j$, 
then by \cite[Lemma 3.4.2]{hir03}, the map from 
$Y \smashprod A \to P$ is also a $T$-equivalence. 
It follows by the two-out-of-three property that 
$i \square j$ is a $T$-equivalence.   
\end{proof}

If the category $\ccal$ is already stably $E$-familiar then the class
$T$ is already contained in the category of weak equivalences. 
Hence so is the set $S$, and $\ccal$ is in fact an 
$L_E \Sp$-model category. 

\begin{corollary}
Let $\dcal$ be a left proper, cellular, spectral model category 
that is stably $E$-familiar. Assume that domain of the generating
cofibrations of $\ccal$ are cofibrant. Then $\ccal$ is an 
$L_E \Sp$-model category.  \qed
\end{corollary}

\section{\texorpdfstring{$E$--}{E-}Familiarisation of stable model categories}\label{sec:stablefamiliarisation}

We now want to consider more model categories that are not necessarily spectral. 
Consider a proper and cellular stable model category 
$\ccal$. By \cite[Theorem 7.2]{BarRoi12}
$\ccal$ is Quillen equivalent to a spectral model category, namely 
the category $\dcal = \Sp^\Sigma(s\ccal)$ of symmetric spectra in simplicial objects in $\ccal$  
equipped with a non-standard model structure. 
Hence there is a Quillen equivalence which by abuse of notation we call
\[
\Sigma^\infty: \ccal \lradjunction \dcal= \Sp^\Sigma(s\ccal): \Omega^\infty
\]
This model category 
$\dcal$ is also proper and cellular. Furthermore, 
if the generating cofibrations for $\ccal$ have cofibrant
domains, then so do the generating cofibrations for $\dcal$.

\begin{theorem}\label{thm:Eexistence}
Let $\ccal$ be a stable, proper and cellular model category, 
such that the domains of the generating cofibrations 
of $\ccal$ are cofibrant. 
Then define $\ccal_E$ to be the left Bousfield
localisation of $\ccal$ at the set of maps
$\Omega^\infty \fibrep S$, where 
$\fibrep$ denotes fibrant replacement in $\Sp^\Sigma(s\ccal)$ 
and where $S = I_\dcal \square J_E$ as above. 
Then 
\begin{enumerate}
\item $\ccal_E$ is stably $E$-familiar, 
\item the weak equivalences of $\ccal_E$ are the 
$T'$-equivalences, for $T'$ the class below  
\[
T'=\{ X \smashprod^L f \,\,|\,\, X \in \ccal, \, \, f \mbox{ is an $E$-equivalence of spectra}\}
\]
\item if $F \co \ccal \to \ecal$ is a left Quillen functor and 
$\ecal$ is stably $E$-familiar, then $F$ factors over 
$\ccal \to \ccal_E$.
\end{enumerate}
\end{theorem}

\begin{proof}
The model categories $\ccal$ and $\dcal=\Sp^\Sigma(s\ccal)$ are Quillen equivalent. Hence \cite[Theorem 3.3.20]{hir03} tell us that the adjunction
\[
\Sigma^\infty: \ccal \lradjunction \dcal: \Omega^\infty.
\] 
induces a 
Quillen equivalence between $L_{\Omega^\infty \fibrep S} \ccal$ and    
$L_{S} \dcal$. 
Thus by \cite[Lemma 7.10]{BarRoi12} 
$\ccal_E = L_{\Omega^\infty \fibrep S} \ccal$
is stably $E$-familiar. 

We may also conclude that the left derived functor of $\Sigma^\infty$
induces an bijection between the weak equivalences of $\ccal_E$
(considered as a class in $\Ho \ccal$) and the $S$-equivalences of $\Ho \dcal$. 
Proposition \ref{prop:classofequivalences} tells us that 
the class of $S$-equivalences in $\dcal$ is equal to the class of $T$-equivalences
where
\[
T=\{ X \smashprod^L f \,\,|\,\, X \in \dcal, \, \, f \mbox{ is an $E$-equivalence of spectra}\}
\]
Consider the class of maps 
\[
T'=\{ X \smashprod^L f \,\,|\,\, X \in \ccal, \, \, f \mbox{ is an $E$-equivalence of spectra}\}
\]
Let $\mathbb{L} \Sigma^\infty$ and $\mathbb{R} \Omega^\infty$ denote the left and right derived functors of $\Sigma^\infty$ and $\Omega^\infty$ respectively.
By \cite[Theorem 6.3]{Len12} 
\[
\mathbb{L} \Sigma^\infty (X \smashprod^L f)
= 
(\mathbb{L} \Sigma^\infty X) \smashprod^L f.
\]
Hence $\mathbb{L} \Sigma^\infty$ takes elements of 
$T'$ to elements of $T$. 
Consider some element $Y \smashprod^L f$ of $T$.
This is weakly equivalent to 
\[
(\mathbb{L}\Sigma^\infty \mathbb{R} \Omega^\infty Y) \smashprod^L f)
\]
and hence is in $\mathbb{L}\Sigma^\infty T'$. 
Thus the derived functor of 
$\Sigma^\infty$ induces a bijection between the class $T'$ and 
the class $T$ up to weak equivalence. 
As a consequence the derived functor of $\Sigma^\infty$ induces a bijection between 
the class of $T'$-equivalences and 
the class of $T$-equivalences. 
It follows that the $T'$-equivalences must be 
the class of weak equivalences of $\ccal_E$. 

For the final point, let $F \co \ccal \to \ecal$ be a 
left Quillen functor. If $\ecal$ is stably $E$-familiar, 
then the left derived functor of $F$ takes 
the $T'$-equivalences to weak equivalences of $\ecal$. 
Hence $F \co \ccal_E \to \ecal$ is also a left Quillen functor. 
\end{proof}

\begin{corollary}
Let $\ccal$ be a stable, proper and cellular model category, 
such that the domains of the generating cofibrations 
of $\ccal$ are cofibrant. Then $\ccal_E$
is the closest stably $E$-familiar model category to $\ccal$. \qed
\end{corollary}

In particular a model category $\ccal$ is stably $E$-familiar if and only if $\ccal_E=\ccal$.

\begin{rmk}
The assumptions on $\ccal$ are more reasonable than they might seem in practice. 
Since we want to perform a left Bousfield localisation, we will 
have to assume that $\ccal$ is left proper and cellular. 
To assume that $\ccal$ is also right proper is not too much of a 
restriction on the kinds of model categories we are able to 
deal with. 

We also need another assumption: that the domains of the generating
cofibrations of $\ccal$ are cofibrant. This is a subtle assumption
that occurs elsewhere in the literature, for example in 
\cite{hov01}. We note that this assumption holds for 
almost all of the cofibrantly generated model categories 
that arise naturally. 
\end{rmk}

It is easy to check that the homotopy mapping spectra for $\ccal_E$ are given by 
the formula below, where $Y_E$ is the fibrant replacement of $Y$ in $\ccal_E$. 
\[
\mathbb{R} \Map_{\ccal_E}(X,Y) = \mathbb{R} \Map_{\ccal}(X, Y_E)
\]
In particular, this mapping spectrum is $E$-local. We can use this to draw some immediate consequences of $E$-familiarisation.

For example, the chromatic Johnson-Wilson theories $E(n)$ satisfy $$L_{E(n-1)}L_{E(n)}=L_{E(n)}$$ \cite[7.5.3]{Rav92}. Thus,

\begin{corollary}
For a proper and cellular stable model category we have $$(\ccal_{E(n)})_{E(n-1)}=\ccal_{E(n-1)}.$$
\qed
\end{corollary}

We can use further use our knowledge of stably $E$-familiar algebraic 
model categories described at the end of Section \ref{sec:framings} 
to read off the following corollaries.

\begin{corollary}
Let $\ccal$ be an algebraic model category and $K(n)$ the $n^{th}$ Morava-$K$-theory for $n \ge 1$. Then $\ccal_{K(n)}$ is trivial. \qed
\end{corollary}

\begin{corollary}
Let $\ccal$ be an algebraic model category and let $E(n)$ denote the $n^{th}$ chromatic Johnson-Wilson spectrum. Then $\ccal_{E(n)} = \ccal_{H\mathbb{Q}}$. \qed
\end{corollary}

If we assume that localisation at $E$ is smashing, we can obtain a nicer description 
of the weak equivalences of $\ccal_E$: in the smashing case $\ccal_E$ is precisely the ``naive'' localisation of $\ccal$ at $L_ES^0$ as motivated in the introduction of this section. Thus, we also obtain that 
\[
\ccal_E=\ccal_{L_ES^0}.
\]
 However, 
for a general model category $\ccal$ and smashing $E$ it is unclear whether this also implies that
the model categories $L_E\ccal$ and $L_{L_ES^0}\ccal$ are Quillen equivalent.

\begin{lemma}
In addition to the assumptions of Theorem \ref{thm:Eexistence}, 
assume  that localisation at $E$ is smashing. Then a map $f$ in $\ccal_E$ is a weak equivalence if and only if
$f \smashprod^L L_E S^0$ is a weak equivalence in $\ccal$.
Hence the weak equivalences of 
$\ccal_E$ are precisely the $L_E S^0$-equivalences. 
\end{lemma}

\begin{proof}
We first show the statement for a spectral model category $\dcal$.
Recall the model category $L_S \dcal$ for $S$ the set 
$I_{\ccal} \square J_E$ from the previous section. We will show that the $S$-equivalences are
precisely the $L_E S^0$-equivalences of $\dcal$. 

Every map in the set $S$ is 
an $L_E S^0$-equivalence, hence every $S$-equivalence is a 
$L_E S^0$-equivalence. 
Now take some $L_E S^0$-equivalence $f \co X \to Y$ in $\dcal$. 
The map $$X \to X \smashprod^L L_E S^0$$ 
is a $T$-equivalence and hence an $S$-equivalence. 
Thus the commutative square
\[
\xymatrix@C+1cm{
X \ar[r]^f \ar[d] & 
Y \ar[d] \\
X \smashprod^L L_E S^0
\ar[r]^{f \smashprod^L L_E S^0} & 
Y \smashprod^L L_E S^0
}
\]
shows that $f$ is $S$-equivalent to a weak equivalence in $\dcal$. 
Weak equivalences in $\dcal$ are in particular $S$-equivalences, so by the 2-out-of-3 axiom of model categories, $f$ must be an $S$-equivalence. 

To move this result from $\dcal$ to $\ccal$ 
we use a similar argument to that of 
the second point of Theorem \ref{thm:Eexistence}. 
The Quillen equivalence 
$(\Sigma^\infty, \Omega^\infty)$ takes the 
$L_E S^0$-equivalences of $\dcal$
bijectively to the $L_E S^0$-equivalences of $\ccal$. 
It follows that the $L_E S^0$-equivalences of $\ccal$ are precisely the 
weak equivalences of $\ccal_E$. 
\end{proof}

The following corollary shows that stable $E$-familiarisation restricts to $E$-localisation in the case of spectra. This shows that the notion of $\ccal_E$ is indeed a good candidate for an analogue of $E$-localisation of a general $\ccal$.

\begin{corollary}\label{cor:Elocalmodule}
Consider the category of modules over a ring spectrum $R$. Then 
\[
(R\leftmod)_E=L_E(R\leftmod)
\]
and in particular $$\Sp_E=L_E\Sp.$$ \qed
\end{corollary}

We can now give a a simple proof
that stable $E$-familiarisation preserves Quillen equivalences. 

\begin{proposition}\label{prop:quillenequivalence}
Let $\ccal$ and $\ecal$ be proper, cellular and stable model categories
such that the domains of their generating cofibrations are cofibrant. 
Let $$F : \ccal \lradjunction \ecal : G$$ be a Quillen equivalence. Then 
there is a Quillen equivalence between the $E$-familiarised model categories
$$F : \ccal_E \lradjunction \ecal_E : G.$$
\end{proposition}

\begin{proof}
Composing $F$ with the identity on $\ecal$ gives us a left Quillen functor $$F \co \ccal \to \ecal_E$$ and $\ecal_E$ is of course stably $E$-familiar.
Hence by the universal property of $\ccal_E$ proved in Theorem \ref{thm:Eexistence} we have a left Quillen functor 
$F \co \ccal_E \to \ecal_E$. 
We now need to 
show that gives us a Quillen equivalence. We do so using
Proposition \ref{prop:classofequivalences} and the method of the second part of the proof of Theorem \ref{thm:Eexistence}. 

Let $T$ be the class of maps
\[
T=\{ A \smashprod^L f \,\,|\,\, A \in \ccal,\, f\, \mbox{is an $E$-equivalence of spectra}\}.
\]
Similarly, let $T'$ be the class of maps 
\[
T'=\{ B \smashprod^L f \,\,|\,\, B \in \ecal,\, f\, \mbox{is an $E$-equivalence of spectra}\}.
\]
Then $\ccal_E = L_{T} \ccal$ and $\ecal_E = L_{T'} \ecal$. 
Let $\mathbb{L}F$ and $\mathbb{R}G$ denote the left and right derived functors of $F$ and $G$ respectively.
By \cite[Theorem 3.3.20]{hir03}, the adjunction $(F,G)$ induces
a Quillen equivalence between $L_T \ccal$ and $L_{\mathbb{L}F(T)} \ecal$. 
But the set ${\mathbb{L}F(T)}$ is isomorphic in $\Ho \ecal$ to the set
$T'$ because Quillen equivalences induce equivalences
of $\Ho(\Sp)$-module categories \cite[Theorem 6.3]{Len12}.
\end{proof}

\section{Examples}\label{sec:examples}

Let $\ccal$ be a spectral model category, 
such that the domains of its generating cofibrations are cofibrant. 
(Recall from \cite[Theorem 7.2]{BarRoi12} that 
any stable, proper and cellular model category is Quillen equivalent
to a spectral one.)
Assume that $\ccal$ has a set of 
compact generators for its homotopy category, 
\cite[Definition 2.1.2]{ss03stabmodcat}. 
Schwede and Shipley prove in the above-mentioned paper that any such 
$\ccal$ is Quillen equivalent to a category $\rightmod\ecal$ where $\ecal$ can be thought of as a ``ring spectrum with several objects''. In the case of $\ccal$ having a single compact generator, $\ecal$ is simply a ring spectrum. 

Let us briefly recap some of the definitions and 
constructions of that result. 
Let $\mathcal{G}$ denote the set of generators of $\ccal$. Then the $\Sp$-enriched category $\ecal$ is simply defined as the full $\Sp$-enriched subcategory of $\ccal$ with objects $\mathcal{G}$. An object $M \in \rightmod\ecal$ consists of a spectrum $M(G)$ for each $G \in \mathcal{G}$ plus morphisms of spectra
\[
M(G) \wedge \ecal(G', G) \longrightarrow M(G') \,\,\,\mbox{for}\,\,\, G, G' \in \mathcal{G}
\]
satisfying certain coherence conditions. By adjunction, such an $M$ is the same as a contravariant spectral functor from $\ecal$ to $\Sp$. The model structure on $\rightmod\ecal$ is has 
weak equivalences and fibrations defined objectwise
\cite[Theorem A.1.1]{ss03stabmodcat}: meaning that a natural transformation $f: M \longrightarrow N$ is a weak equivalence or a fibration if and only if 
\[
f_G \co M(G) \longrightarrow N(G)
\] 
is so for each $G \in \gcal$. Theorem 3.9.3 of \cite{ss03stabmodcat} then describes a Quillen equivalence
\[
\Hom(\mathcal{G},-): \ccal \rladjunction \rightmod\ecal: - \wedge_{\ecal} \mathcal{G}
\]
for spectral $\ccal$.

\bigskip
This is a highly useful description of a stable model category
and we would like to obtain a description of the $E$-familiarisation $\ccal_E$ of $\ccal$ in terms of $\rightmod\ecal$. We note that this is a rather special case as not every stable model category has a set of compact generators \cite[Corollary B.13]{HovStr99}.

By Proposition \ref{prop:quillenequivalence} we know that $\ccal_E$ and $(\rightmod \ecal)_E$ are Quillen equivalent,
so we shall find another description of $(\rightmod \ecal)_E$. 

\bigskip
The generating cofibrations of this model category have the form 
\[
\hom(-,g) \smashprod (i \square j)
\] 
where $g$ is a cofibrant 
and fibrant replacement of one of the 
compact generators for $\ccal$, 
$i$ is a generating cofibration for $\Sp$ and $j$ is a generating
acyclic cofibration for $L_E \Sp$. 

We can make another model structure on $\rightmod \ecal$
by taking the same cofibrations as before, but taking the generating
set of acyclic cofibrations to be those maps of form
\[
\hom(-,g) \smashprod j
\]
for $g$ a generator and $j$ a generating
acyclic cofibration for $L_E \Sp$. We shall call this set of maps
$K$. One can either check directly that these sets give a model structure or one can
alter \cite[Theorem A.1.1]{ss03stabmodcat} to use $L_E \Sp$ instead of $\Sp$. 

We claim that this model structure equals the model structure of $(\rightmod\ecal)_E$. An element of $K$ 
can be described as 
\[
\hom(-,g) \smashprod \left( (\ast \to S^0) \square j \right)
\] 
Hence every element of $K$ is an acyclic cofibration of $(\rightmod \ecal)_E$.
Conversely, since $E$-acyclic cofibrations of spectra are closed under pushout along
cofibrations, it follows that the set of acyclic cofibrations generated by 
the set $K$ is precisely the set of acyclic cofibrations of $(\rightmod \ecal)_E$. Thus we have shown the following.

\begin{proposition}
The model category $(\rightmod \ecal)_E$ is the category of contravariant
spectral functors from $\ecal$ to $L_E \Sp$, equipped with the 
model structure where fibrations and weak equivalences are defined objectwise. 
Thus the fibrant objects are those functors $M$ such that 
$M(g)$ is fibrant in $\Sp$ and $E$-local for all $g \in \ecal$. \qed
\end{proposition}

Consider the case where $\ccal$ has a single compact generator. 
Following the above we can replace this by a category of functors to $\Sp$. 
Indeed, \cite[Theorem 3.1.1]{ss03stabmodcat} states that 
$\ccal$ is Quillen equivalent to the category of $R$--modules, $\rightmod R$, 
for some ring spectrum $R$. Hence we are essentially in 
the same situation as Corollary \ref{cor:Elocalmodule}. Our work above
recovers the well-known result that $(\rightmod R)_E$ is the category of $R$-modules, with weak equivalences the
underlying $E$-equivalences of spectra.

\section{\texorpdfstring{$E$--}{E-}familiarisation and homotopy pushouts}\label{sec:pushouts}

We want to give another description of $\ccal_E$ via a universal property. 
We will relate $\ccal_E$ to a pushout of model categories. 
While the pullback of model categories is well-understood, \cite{berg11}, 
the pushout is more complicated and is not often used. Roughly speaking, the homotopy pushout of a corner diagram of Quillen adjunctions
\[
\ccal \rladjunction \dcal \lradjunction \ecal
\]
is supposed to be a model category $\mathcal{P}$ that satisfies a universal property analogous to the pushout of a diagram within a category. Unfortunately, the homotopy-theoretic pushout construction is rather delicate and its existence and description not always clear.

\bigskip
However there is a special case where we can construct pushouts of model categories and verify that they have the correct universal property. 
By working in a particular context, we avoid the general question of whether 
homotopy pushouts of model categories exist in general. 

Let $\mcal_2$ be a left Bousfield localisation of $\mcal_1$ at a class of maps $W$. 
Without loss of generality we assume that the maps in $W$ are morphisms between cofibrant objects.
In particular, this gives us a Quillen pair \[
\id \co \mcal_1 \lradjunction \mcal_2 = L_W \mcal_1: \id
\]
Assume that we have a Quillen adjunction 
\[
F : \mcal_1 \lradjunction \ncal_1 : G
\]
We are now going to discuss the homotopy pushout of the corner diagram below for this special case
\[
L_W\mcal_1=\mcal_2 \rladjunction \mcal_1 \lradjunction \ncal_1.
\]

\begin{definition}\label{def:pushout}
The homotopy pushout of the above diagram is defined, if it exists, as the Bousfield localisation $L_{FW} \ncal_1$ of $\ncal_1$.
\end{definition}

To justify this definition we need to see that $\ncal_2=L_{FW}\ncal_1$ 
(provided it exists) satisfies the desired properties that a homotopy pushout is supposed to have.
First we note that by \cite[Theorem 3.3.20]{hir03} $F$ and $G$ induce a Quillen adjunction
\[
F :\mcal_2 \lradjunction \ncal_2 : G
\] 

Assume that there is a model category $\dcal$ with Quillen adjunctions 
\[
\begin{array}{rcl}
\mcal_2 & \lradjunction & \dcal \\
F' : \ncal_1 & \lradjunction & \dcal : G'
\end{array}
\]
such that in the diagram below, the two different composites 
of left adjoints from $\mcal_1$ to $\dcal$ agree up to natural isomorphism. 
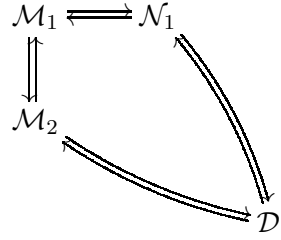
\[
\xymatrix{ 
\mcal_1 \ar@<0.5ex>[r] \ar@<-0.5ex>[d] & 
\ncal_1 \ar@<0.5ex>@/^/[ddr] \ar[l] & \\
\mcal_2 \ar[u] \ar@<0.5ex>@/_/[drr] & & \\
& & \dcal \ar@/^/[ull] \ar@/_/[uul].
}
\]
Because the vertical functors in the square below are simply identity functors
it follows immediately that we may add $\ncal_2$ and obtain a 
commutative diagram of adjoint pairs. 
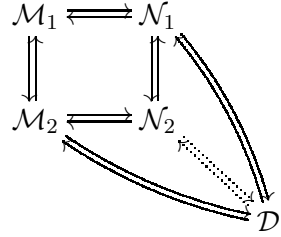
\[
\xymatrix{ 
\mcal_1 \ar@<0.5ex>[r] \ar@<-0.5ex>[d] & 
\ncal_1 \ar@<0.5ex>@/^/[ddr] \ar[l] \ar@<-0.5ex>[d]& \\
\mcal_2 \ar[u] \ar@<0.5ex>[r] \ar@<0.5ex>@/_/[drr] & 
\ncal_2 \ar[l] \ar[u] \ar@<0.5ex>@{.>}[dr]& \\
& & \dcal \ar@/^/[ull] \ar@/_/[uul] \ar@{.>}[ul]
}
\]
We must check that the adjunction below is Quillen adjunction. 
\[
F': \ncal_2 \lradjunction \dcal: G'
\]
The model category $\ncal_2$ is the Bousfield localisation of $\ncal_1$ with respect to the class of maps $Ff$ where $f$ is a weak equivalence between cofibrant objects of $\mcal_2$. 
Thus $(F'\circ F)(f)$ is a weak equivalence in $\dcal$. This means that $F'$ uniquely factors over $\ncal_2$. Furthermore, by construction, $\ncal_2$, if it exists, is unique
up to Quillen equivalence.

\bigskip
Recall that the stable $E$-familiarisation $\ccal_E$ satisfies the following universal property. Given a left Quillen functor $F: \ccal \longrightarrow \dcal$ with $\dcal$ stably $E$-familiar, $F$ also gives rise to a left Quillen functor $C_E \longrightarrow \dcal$ via
\[
\xymatrix{ \ccal \ar[rr]^{F} \ar[d]_{id} & & \dcal. \\
\ccal_E \ar[urr] & &
}
\]
This fact allows us to relate $\ccal_E$ and certain homotopy pushouts. Let $X \in \ccal$ be fibrant and cofibrant. Then we have a Quillen adjunction
\[
X \smashprod -: \Sp \lradjunction \ccal: \Map(X,-).
\]

Using Definition \ref{def:pushout} we can read off the following for a proper and cellular stable model category $\ccal$.
\begin{lemma}
The homotopy pushout $\Pu_X$ of the diagram 
\[
L_E\Sp \rladjunction \Sp \lradjunction \ccal
\]
exists and is the Bousfield localisation of $\ccal$ with respect to the set of maps
below, where $J_E$ is the set of generating acyclic cofibrations
of $L_E \Sp$. 
\[
X \smashprod^L J_E = 
\{ X \smashprod^L j \ | \ j \in J_E
\}
\] \qed
\end{lemma}

So in particular we know that this homotopy pushout exists.
Because $\ccal_E$ is stably $E$-familiar we have a commutative square of Quillen adjunctions
\[
\xymatrix{ \Sp \ar@<0.5ex>[r] \ar@<-0.5ex>[d] & \ccal \ar[l]\ar@<-0.5ex>[d] \\
L_E\Sp \ar[u] \ar@<0.5ex>[r] & \ccal_E \ar[u] \ar[l].
}
\]
By the universal property of $\Pu_X$, there is a Quillen adjunction $\Pu_X \lradjunction \ccal_E$ for each $X$. We can show that $\ccal_E$ is the ``closest'' model category to those pushouts in the following sense.

\begin{theorem}
The Quillen adjunction
\[
\ccal \lradjunction \ccal_E
\]
factors over $$\Pu_X \lradjunction \ccal_E$$ for all fibrant-cofibrant $X \in \ccal$. If there is any other stable $\dcal$ with a Quillen adjunction $$F: \ccal \lradjunction \dcal: G$$ that factors over $$\Pu_X \lradjunction \dcal$$ for all fibrant-cofibrant $X$, then $(F,G)$ also factors over $\ccal_E$.
\end{theorem}

\begin{proof}
The pushout $\Pu_X$ is defined as the Bousfield localisation of $\ccal$ at the set of maps $X \smashprod j$ with $j \in J_E$. 
By Proposition \ref{prop:classofequivalences} we know that 
$\ccal_E$ is the localisation of $\ccal$ at the class of maps of form 
$X \smashprod f$ for $f$ and $E$-equivalence of spectra. 
Thus we see that for every $X \in \ccal$ the identity gives us a Quilen adjunction
\[
\id: \Pu_X \lradjunction \ccal_E: \id
\]
because every weak equivalence in $\Pu_X$ is also a weak equivalence in $\ccal_E$.

If the given Quillen adjunction $(F,G)$ induces a Quillen adjunction
\[
F: \Pu_X \lradjunction \dcal: G,
\]
then $F$ sends all morphisms of the form $X \smashprod^L j$ 
 to weak equivalences in $\dcal$. Thus it induces a Quillen adjunction
\[
\ccal_E \lradjunction \dcal
\]
which is what we wanted to prove.
\end{proof}

\section{Modular rigidity for 
\texorpdfstring{$E$--}{E-}local spectra}\label{sec:Elocal}

We can show that stable frames encode all homotopical information of the $E$-local stable homotopy category. The triangulated structure of $\Ho(L_E\Sp)$ alone is not sufficient for this: given just a triangulated equivalence
\[
\Phi: \Ho(L_E \Sp) \longrightarrow \Ho(\ccal)
\]
for a stable model category $\ccal$ does not imply in general that $L_E\Sp$ and $\ccal$ are Quillen equivalent. In fact, Quillen equivalence can only be deduced from a triangulated equivalence of homotopy categories in some very special cases. To this date, the only nontrivial cases known of this `rigidity' are the stable homotopy category itself \cite{Sch07} and the case $E = K_{(2)}$ \cite{Roi07}. However, if do not only have a triangulated equivalence as above but also assume that this equivalence is a $\Ho(\Sp)$-module equivalence, we can show that $L_E\Sp$ and $\ccal$ are Quillen equivalent. 

For the case of $E$ smashing, the result below appeared as \cite[Theorem 9.5]{BarRoi11b}, but with a slight modification it would also work for all $E$ such that $\Ho(L_E\Sp)$ has one compact generator, e.g. $E=K(n)$. The proof relied on the result of Schwede and Shipley \cite[Theorem 3.1.1]{ss03stabmodcat} that every stable model category $\ccal$ such that $\Ho(\ccal)$ has one compact generator is Quillen equivalent to the category of modules over some ring spectrum. However, for some $E$ the $E$-local stable homotopy category does not possess any compact objects at all \cite[Corollary B.13]{HovStr99}, let alone a set of compact generators. But in fact a more general result is true.

\begin{theorem}\label{thm:elocalrigid}
Let $\ccal$ be a stable model category.
Assume we have an equivalence of triangulated categories
\[
\Phi: \Ho(L_E\Sp) \longrightarrow \Ho(\ccal)
\]
then $L_E\Sp$ and $\ccal$ are Quillen equivalent if and only if $\Phi$ is an equivalence of $\Ho(\Sp)$-model categories. 
\end{theorem}

\begin{proof}
The ``only if'' part is true by \cite[Theorem 6.3]{Len12}: a Quillen equivalence induces a $\Ho(\Sp)$-module equivalence. 

Now let us assume that we have a $\Ho(\Sp)$-module equivalence
\[
\Phi: \Ho(L_E\Sp) \longrightarrow \Ho(\ccal)
\] 
It follows that $\Phi^{-1}$ induces a weak equivalence of homotopy
mapping spectra
\[
\Phi^{-1} \co \mathbb{R} \Map_\ccal(X,Y) 
\longrightarrow 
\mathbb{R} \Map_\ccal(\Phi^{-1} X,\Phi^{-1} Y)
\]
for $X, Y \in \ccal$.
The right-hand-side is an $E$-local spectrum as $L_E\Sp$ is stably $E$-familiar.
Hence every homotopy mapping spectrum of $\ccal$
is $E$-local, so $\ccal$ is stably $E$-familiar
by \cite[Theorem 7.8]{BarRoi11b}. 

Thus for fibrant and cofibrant $X \in \ccal$ , the Quillen functor
\[
X \smashprod -: \Sp \longrightarrow \ccal
\]
factors over $L_E\Sp$ as a Quillen functor
\[
X \smashprod -: L_E \Sp \longrightarrow \ccal.
\]
Now we consider $X$ a cofibrant-fibrant replacement of $\Phi(S^0)$.  Because $\Phi$ is a $\Ho(\Sp)$-module equivalence we see that
\[
X \dsmash (-) = \Phi(L_E S^0) \dsmash (-) = \Phi(L_E S^0 \dsmash - ) = \Phi(-).
\]
This means that $\Phi$ is derived from a Quillen functor. This 
Quillen functor must therefore be a Quillen equivalence, 
which is what we wanted to prove. 
\end{proof}

%\bibliography{modular}
\bibliographystyle{alpha}

\end{document}